\documentclass[num,preprint,12pt]{elsarticle}
\usepackage{a4wide}
\usepackage{hyperref}
\usepackage{natbib}
\usepackage{color}
\usepackage{amsmath,amssymb,bbm,enumerate}
\usepackage{pstricks}
\usepackage{pst-3d,float}

\newtheorem{theorem}{Theorem}[section]

\newtheorem{definition}{Definition}[section]

\newtheorem{lemma}{Lemma}[section]

\newtheorem{proposition}[theorem]{Proposition}

\newenvironment{proof}{\medskip\noindent{\bf Proof.}\;}{\null\hfill $\Box$\par\medskip }

\def\rmd{\mathrm{d}}
\def\rme{\mathrm{e}}
\def\rmi{\mathrm{i}}
\def\1{\mathbbm{1}}
\def\gammalim{\overline{\gamma}}

\begin{document}

\begin{frontmatter}

\title{Large scale behavior of wavelet coefficients of non-linear subordinated processes with long memory}

\date{June 3, 2010}

\author[clausel]{Marianne Clausel}
\ead{clausel@univ-paris12.fr}

\author[francois]{Fran\c{c}ois Roueff\corref{cor1}}
\ead{roueff@telecom-paristech.fr}

\author[murad]{Murad~S. Taqqu}
\ead{murad@math.bu.edu}

\author[ciprian]{Ciprian Tudor}
\ead{Ciprian.Tudor@math.univ-lille1.fr}

\cortext[cor1] {Corresponding author. Tel.: +33 1 45 81 78 53; fax: +33 1 45 81
  71 44.}

\address[clausel]{Laboratoire d'Analyse et de Math\'ematiques Appliqu\'ees, UMR
  8050 du CNRS, Universit\'e Paris Est, 61 Avenue du G\'en\'eral de Gaulle,
  94010 Cr\'eteil Cedex, France.}
\address[francois]{Institut Telecom, Telecom Paris, CNRS LTCI, 46 rue Barrault,
  75634 Paris Cedex 13, France}

\address[murad]{Departement of Mathematics and Statistics, Boston University,
  Boston, MA 02215, USA}

\address[ciprian]{Laboratoire Paul Painlev\'e, UMR 8524 du CNRS, Universit\'e
  Lille 1, 59655 Villeneuve d'Ascq, France. Associate member: SAMM,
  Universit{\'e} de Panth\'eon-Sorbonne Paris 1.}

\begin{abstract}
We study the asymptotic behavior of wavelet coefficients of random processes
with long memory. These processes may be stationary or not and are obtained as
the output of non--linear filter with Gaussian input. The wavelet coefficients
that appear in the limit are random, typically non--Gaussian and belong to a
Wiener chaos. They can be interpreted as wavelet coefficients of a generalized
self-similar process.
\end{abstract}


\end{frontmatter}

\section{Introduction}\label{sec:intro}
Let $X=\{X_n\}_{n\in\mathbb{Z}}$ be a stationary Gaussian process with mean zero, unit variance and spectral density $f(\lambda),\lambda\in (-\pi,\pi]$ and thus covariance equal to
\[
r(n)=\mathbb{E}(X_0 X_n)=\int_{-\pi}^{\pi}\rme^{\rmi n\lambda}f(\lambda)\rmd \lambda\;.
\]
The process $\{X_n\}_{n\in\mathbb{Z}}$ is said to have {\it short memory} or
{\it short--range dependence} if $f(\lambda)$ is bounded around $\lambda=0$ and
{\it long memory} or {\it long--range dependence} if $f(\lambda)\to\infty$ as
$\lambda\to 0$. We will suppose that $\{X_n\}_{n\in\mathbb{Z}}$ has
long--memory with memory parameter $d>0$, that is,
\[
f(\lambda)\sim |\lambda|^{-2d}f^*(\lambda)\mbox{ as }\lambda \to 0
\]
where $f^*(\lambda)$ is a bounded spectral density which is continuous and
positive at the origin. It is convenient to interpret this behavior as the
result of a fractional integrating operation, whose transfer function reads
$\lambda\mapsto (1-\rme^{-\rmi\lambda})^{-d}$. Hence we set
\begin{equation}\label{e:sdf}
f(\lambda)=|1-\rme^{-\rmi\lambda}|^{-2d}f^*(\lambda),\quad\lambda\in (-\pi,\pi]\;.
\end{equation}

We relax the above assumptions in two ways  :
\begin{enumerate}
\item Consider, instead of the Gaussian process $\{X_n\}_{n\in\mathbb{Z}}$ the
  non--Gaussian process $\{G(X_n)\}_{n\in\mathbb{Z}}$ where $G$ is a non--linear
  filter such that $\mathbb{E}[G(X_n)]=0$ and
  $\mathbb{E}[G(X_n)^2]<\infty$. The non--linear process
  $\{G(X_n)\}_{n\in\mathbb{Z}}$ is said
  to be subordinated to the Gaussian process $\{X_n\}_{n\in\mathbb{Z}}$.
\item Drop the stationarity assumption by considering a process
  $\{Y_{n}\}_{n\in\mathbb{Z}}$ which becomes stationary when differenced $K\geq
  0$ times.
\end{enumerate}
We shall thus consider $\{Y_n\}_{n\in\mathbb{Z}}$ such that 
\[
\left(\Delta^K Y\right)_{n}=G(X_n),\quad n\in\mathbb{Z}\;,
\]
where $(\Delta Y)_n=Y_{n}-Y_{n-1}$ and where $\{X_n\}_{n\in\mathbb{Z}}$ is
Gaussian with spectral density $f$ satisfying~(\ref{e:sdf}).

Since $Y=\{Y_n\}_{n\in\mathbb{Z}}$ is random so will be its wavelet
coefficients $\{W_{j,k},\,j\geq 0,\,k\in\mathbb{Z}\}$ which are defined below. Our
goal is to find the distribution of the wavelet coefficients at large scales
$j\to\infty$. This is an important step in developing methods for estimating
the underlying long memory parameter $d$. The large scale behavior of the
wavelet coefficients was studied in \cite{moulines:roueff:taqqu:2007:jtsa} in
the case where there was no filter $G$, that is, when $Y$ is a Gaussian process
such that $\Delta^K Y=X$, and also in the case where $Y$ is a non--Gaussian
linear process (see \cite{roueff-taqqu-2009}).

We obtain our random wavelet coefficients by using more general linear filters
that those related to multiresolution analysis (MRA) (see for
e.g.~\cite{meyer:1990}, \cite{mallat:1998}). In practice, however, the methods
are best implemented using Mallat's algorithm and a MRA. Our filters are
denoted $h_j$ where $j$ is the scale and we use a scaling factor
$\gamma_j\uparrow \infty$ as $j\uparrow\infty$. In the case of a MRA,
$\gamma_j=2^j$ and $h_j$ are generated by a (low pass) scaling filter and its
corresponding quadratic (high pass) mirror filter. More generally one can use a
scaling function $\varphi$ and a mother wavelet $\psi$ to generate the random
wavelet coefficients by setting
\begin{equation}
  \label{eq:classic-wav-set}
  W_{j,k}=\int_{\mathbb{R}}\psi_{j,k}(t)\left(\sum_{\ell\in\mathbb{Z}}\varphi(t-\ell)Y_{\ell}\right)\rmd t,  
\end{equation}
where $\psi_{j,k}=2^{-j/2}\psi(2^{-j}t-k),\,j\geq 0$. Observe that we use here
the engineering convention that large values of $j$ correspond to large scales
and hence low frequencies. If $\varphi$ and $\psi$ have compact support then
the corresponding filters $h_j$ have finite support of size $O(2^j)$. For more
details on related conditions on $\varphi$ and $\psi$
(see~\cite{moulines:roueff:taqqu:2007:jtsa}).

The idea of using wavelets to estimate the long memory coefficient $d$ goes
back to Wornell and al.~(\cite{wornell:oppenheim:1992}) and
Flandrin~(\cite{flandrin:1989a,flandrin:1989b,flandrin:1991a,flandrin:1999}). See
also  Abry and al.~(\cite{abry:veitch:1998,abry:veitch:flandrin:1998}). Those
methods are an alternative to the Fourier methods developed by Fox and
Taqqu~(\cite{fox:taqqu:1986}) and
Robinson~(\cite{robinson:1995:GPH,robinson:1995:GSE}. For a general comparison
of Fourier and wavelet approach, see
\cite{fay-mouline-roueff-taqqu-2009}. The  case of the Rosenblatt
process, which is the Hermite process of order $q=2$, was studied
by~\cite{bardet:tudor:2009}.

The paper is structured as follows. In Section~\ref{s:assump}, we introduce the
wavelet filters. The processes are defined in Section~\ref{sec:intrep} using
integral representations and Section~\ref{sec:wienerchaos} presents the
so--called Wiener chaos decomposition. The main result and its interpretations
is given in Section~\ref{sec:WC}. It is proved in
Section~\ref{sec:proofth}. Auxiliary lemmas are presented and proved in
Sections~\ref{sec:appendix} and \ref{sec:techlemma}.

\section{Assumptions on the wavelet filter}\label{s:assump}

The wavelet transform of $Y$ involves the application of a linear filter
$h_j(\tau),\tau\in\mathbb{Z}$, at each scale $j\geq 0$. We shall characterize
the filters $h_j$ by their discrete Fourier transform~:
\[
\widehat{h}_j(\lambda)=\sum_{\tau\in\mathbb{Z}}h_j(\tau) \rme^{-\rmi \lambda\tau },\,\lambda\in [-\pi,\pi]\;.
\]
Assumptions on $\widehat{h}_j$ are stated below. The resulting wavelet coefficients are defined as
\[
W_{j,k}=\sum_{\ell\in\mathbb{Z}}h_j(\gamma_j k-\ell)Y_{\ell},
\quad j\geq0,\;k\in\mathbb{Z}\;,
\]
where $\gamma_j\uparrow \infty$ is a sequence of non--negative scale factors
applied at scale $j$, for example $\gamma_j=2^j$. We will assume that for any
$m\in\mathbb{Z}$,
\begin{equation}\label{eq:scale}
\lim_{j\to\infty}\frac{\gamma_{j+m}}{\gamma_j}=\gammalim_m>0\;.
\end{equation}

As noted, in this paper, we do not assume that the wavelet coefficients are
orthogonal nor that they are generated by a multiresolution analysis. Our
assumptions on the filters $h_j$ are as follows~:
\begin{description}
\item {a. }\underline{Finite support}: For each $j$,
  $\{h_j(\tau)\}_{\tau\in\mathbb{Z}}$ has finite support.
\item {b. }\underline{Uniform smoothness}: There exists $M\geq K$, $\alpha>1/2$
  and $C>0$ such that for all $j\geq0$ and $\lambda\in [-\pi,\pi]$,
    \begin{equation}\label{EqMajoHj}
    |\widehat{h}_j(\lambda)|\leq \frac{C\gamma_j^{1/2}|\gamma_j\lambda|^M}{(1+\gamma_j|\lambda|)^{M+\alpha}}\;.
    \end{equation}
By $2\pi$-periodicity of $\widehat{h}_j$ this inequality can be extended to
$\lambda\in\mathbb{R}$ as
\begin{equation}\label{EqMajoHjR}
|\widehat{h}_{j}(\lambda)|\leq C
\frac{\gamma_j^{1/2}|\gamma_j\{\lambda\}|^M}
{(1+\gamma_j|\{\lambda\}|)^{\alpha+M}}\;.
\end{equation}
where $\{\lambda\}$ denotes the element of $(-\pi,\pi]$
such that $\lambda-\{\lambda\}\in2\pi\mathbb{Z}$.
\item {c. }\underline{Asymptotic behavior}: There exists some non identically
  zero function $ \widehat{h}_{\infty}$ such that for any
  $\lambda\in\mathbb{R}$,
\begin{equation}\label{EqLimHj}
\lim_{j \to +\infty}(\gamma_j^{-1/2}\widehat{h}_{j}(\gamma_j^{-1}\lambda))= \widehat{h}_{\infty}(\lambda)\;.
\end{equation}
\end{description}
Observe that while $\widehat{h}_j$ is $2\pi$-periodic, the function
$\widehat{h}_{\infty}$ is a non-periodic function on $\mathbb{R}$ (this
follows from~(\ref{EqMajoHinf}) below). For the connection between these
assumptions on $h_j$ and corresponding assumptions on the scaling function
$\varphi$ and the mother wavelet $\psi$ in the classical wavelet
setting~(\ref{eq:classic-wav-set})
(see~\cite{moulines:roueff:taqqu:2007:jtsa}). In particular, in that case, one
has $\widehat{h}_{\infty}=\widehat{\varphi}(0)\overline{\widehat{\psi}}$.

Our goal is to study the large scale behavior of the random wavelet coefficients
\begin{equation}\label{e:W1}
W_{j,k}=
\sum_{\ell\in\mathbb{Z}}h_j(\gamma_j k-\ell)Y_{\ell}
=\sum_{\ell\in\mathbb{Z}}h_j(\gamma_j k-\ell)\left(\Delta^{-K}G(X)\right)_{\ell},
\end{equation}
where we set symbolically 
$Y_{\ell}=\left(\Delta^{-K}G(X)\right)_{\ell}$ 
for $(\Delta^K Y)_{\ell}=G(X_{\ell})$.

By Assumption~(\ref{EqMajoHj}), $h_j$ has null moments up to order $M-1$, that
is, for any $m\in\{0,\cdots,M-1\}$,
\begin{equation}\label{EqMom1}
\sum_{\ell\in \mathbb{Z}}h_j(\ell)\ell^m =0\;.
\end{equation}
Therefore, since $M\geq K$, $\widehat{h}_j$ can be expressed as
\begin{equation}\label{EqHjkVSHj}
\widehat{h}_j(\lambda)=(1-\rme^{-{\rmi}\lambda})^K \widehat{h}_j^{(K)}(\lambda),
\end{equation}
where $\widehat{h}_j^{(K)}$ is also a trigonometric polynomial of the form
\begin{equation}\label{EqHjk}
\widehat{h}_j^{(K)}(\lambda)
=\sum_{\tau\in\mathbb{Z}} h_j^{(K)}(\tau)\rme^{-\rmi\lambda\tau},
\end{equation}
since $h_j^{(K)}$ has finite support for any $j$. Then
we obtain another way of expressing $W_{j,k}$, namely,
\begin{equation}\label{e:W2}
W_{j,k}=\sum_{\ell\in\mathbb{Z}}h_j^{(K)}(\gamma_j k-\ell)G(X_{\ell})\;.
\end{equation}
We have thus incorporated the linear filter $\Delta^{-K}$ in~(\ref{e:W1}) into the filter $h_j$ and denoted the new filter $h_j^{(K)}$.\\

{\bf Remarks}
\begin{enumerate}
\item Since $\{G(X_\ell),\ell\in\mathbb{Z}\}$ is stationary,
  it follows from~(\ref{e:W2}) that $\{W_{j,k},k\in\mathbb{Z}\}$ is stationary
  for each scale $j$. 
\item Observe that $\Delta^K Y$ is centered by definition. However, by
  (\ref{EqMom1}), the definition of $W_{j,k}$ only depends on $\Delta^M Y$. In
  particular, provided that $M\geq K+1$, its value is not modified if a
  constant is added to $\Delta^K Y$, whenever $M\geq K+1$.
\item\label{remitem:TFhinfty} Assumptions (\ref{EqMajoHj}) and~(\ref{EqLimHj})
  imply that for any $\lambda\in\mathbb{R}$,
\begin{equation}\label{EqMajoHinf}
|\widehat{h}_{\infty}(\lambda)|\leq C\frac{|\lambda|^M}{(1+|\lambda|)^{\alpha+M}}\;.
\end{equation}
Hence $\widehat{h}_{\infty}\in L^{2}(\mathbb{R})$ since $\alpha>1/2$. 
\item The Fourier transform of $f$,
\begin{equation}\label{eq:TFdef}
\mathfrak{F}(f)(\xi)=\int_{\mathbb{R}^q} f(t)\rme^{-\rmi t^T\xi}\;\rmd^q t,\quad \xi\in\mathbb{R}^q\;,
\end{equation}
is defined for any $f\in L^2(\mathbb{R}^{q},\mathbb{C})$. We let $h_{\infty}$ be the
$L^2(\mathbb{R})$ function such that $\widehat{h}_{\infty}=\mathfrak{F}[h_{\infty}]$.
\end{enumerate}

\section{Integral representations}\label{sec:intrep}
It is convenient to use an integral representation in the spectral domain to
represent the random processes (see for
example~\cite{major:1984,nualart:2006}). The stationary Gaussian process
$\{X_k,k\in\mathbb{Z}\}$ with spectral density~(\ref{e:sdf}) can be written as
\begin{equation}\label{e:intrepX}
X_\ell=\int_{-\pi}^{\pi}\rme^{\rmi\lambda
  \ell}f^{1/2}(\lambda)\rmd\widehat{W}(\lambda)=\int_{-\pi}^{\pi}\frac{\rme^{\rmi\lambda
    \ell}f^{*1/2}(\lambda)}{|1-\rme^{-{\rmi}\lambda}|^d}\rmd\widehat{W}(\lambda),\quad\ell\in\mathbb{N}\;.
\end{equation}
This is a special case of
\begin{equation}\label{e:int}
\widehat{I}(g)=\int_{\mathbb{R}}g(x)\rmd\widehat{W}(x),
\end{equation}
where $\widehat{W}(\cdot)$ is a complex--valued Gaussian random measure satisfying
\begin{eqnarray}\label{e:w1}
\mathbb{E}(\widehat{W}(A))&=&0\quad\text{ for every Borel set $A$ in }\mathbb{R}\;,\\
\label{e:w2}
\mathbb{E}(\widehat{W}(A)\overline{\widehat{W}(B)})&=&|A\cap B|
\text{ for every Borel sets $A$ and $B$ in }\mathbb{R}\;,\\
\label{e:w3}
\sum_{j=1}^{n}\widehat{W}(A_j)&=&\widehat{W}(\bigcup_{j=1}^n A_j)\text{ if
}A_1,\cdots,A_n\mbox{ are disjoint Borel sets in }\mathbb{R}\;,\\
\label{e:w4}
\widehat{W}(A)&=&\overline{\widehat{W}(-A)}\quad\text{ for every Borel set $A$ in }\mathbb{R}\;.
\end{eqnarray}
The integral~(\ref{e:int}) is defined for any function $g\in L^2(\mathbb{R})$ and one has the isometry
\[
\mathbb{E}(|\widehat{I}(g)|^2)=\int_{\mathbb{R}}|g(x)|^2\rmd x\;.
\]
The integral $\widehat{I}(g)$, moreover, is real--valued if
\[
g(x)=\overline{g(-x)}\;.
\]

We shall also consider multiple It\^{o}--Wiener integrals
\[
\widehat{I}_q(g)=\int^{''}_{\mathbb{R}^q}g(\lambda_1,\cdots,\lambda_q)\rmd \widehat{W}(\lambda_1)\cdots\rmd\widehat{W}(\lambda_q)
\]
where the double prime indicates that one does not integrate on hyperdiagonals $\lambda_i=\pm
\lambda_j,i\neq j$. The integrals $\widehat{I}_q(g)$ are handy because we will be able to expand our non--linear functions $G(X_k)$ introduced in Section~\ref{sec:intro} in multiple integrals of this type.

These multiples integrals are defined as follows. Denote by
$\overline{L^2}(\mathbb{R}^{q},\mathbb{C})$ the space of complex valued
functions defined on $\mathbb{R}^{q}$ satisfying
\begin{gather}\label{e:antisym}
g(-x_{1},\cdots,-x_{q})= \overline{g(x_{1},\cdots, x_{q})}\mbox{ for }(x_{1},\cdots, x_{q}) \in \mathbb{R}^q\;,\\
\label{e:fL2}
\Vert g\Vert ^{2} _{L^{2}}:= \int_{\mathbb{R}^{q}} \left| g(x_{1},\cdots, x_{q}) \right| ^{2} \rmd x_{1}\cdots \rmd x_{q} <\infty\;.
\end{gather}
Let $\tilde{L}^2(\mathbb{R}^{q},\mathbb{C})$ denote the set of functions in
$\overline{L^2}(\mathbb{R}^{q},\mathbb{C})$ that are symmetric in the sense
that $g=\tilde{g}$ where $\tilde{g} (x_{1},\cdots, x_{q}) =1/q!\sum_{\sigma}
g(x_{\sigma (1)},\cdots, x_{\sigma (q)})$, where the sum is over all
permutations of $\{1,\dots,q\}$.  One defines now the multiple integral with
respect to the spectral measure $\widehat{W}$ by a density argument. For a step
function of the form
\begin{equation*}
g = \sum_{j_{\ell}=\pm 1,\cdots, \pm N} c_{j_{1},\cdots, j_{n}} 1_{\Delta _{j_{1}} }\times\cdots\times 1_{\Delta _{j_{n}} }
\end{equation*}
where the $c$'s are real--valued, $\Delta _{j_{\ell}}= -\Delta _{-j_{\ell}}$ and $\Delta _{j_{\ell}}\cap \Delta _{j_{m}}=\emptyset$ if $\ell\not= m$, one sets
\begin{equation}
\label{hatI1}
\widehat{I}_{q} (g)= \sum_{j_{\ell}=\pm 1,\cdots, \pm N} {\!\!\!\!\!\!\!\!\!\!}^{''} c_{j_{1},\cdots, j_{n}}\widehat{W}(\Delta _ {j_{1}})\cdots \widehat{W}(\Delta _{j_{n}})\;.
\end{equation}
Here, $\sum^{\prime\prime}$ indicates that one does not sum over the
hyperdiagonals, that is, when $j_{\ell}=\pm j_{m}$ for $\ell\neq m$. The
integral $\widehat{I}_q$ verifies that
\begin{equation}\label{e:cov-multiple}
\mathbb{E}(\widehat{I}_{q}(g_1) \widehat{I}_{q'}(g_2))=
\left\{
\begin{array}{l}
q! \langle g_1, g_2\rangle _{L^{2}}, \mbox{ if } q=q'\\
0, \mbox{ if } q\neq q'.
\end{array}
\right.
\end{equation}
Observe, moreover, that for every step function $g$ with $q$ variables as above
$$
\widehat{I}_{q}(g)= \widehat{I}_{q}( \tilde{g}).
$$
Since the set of step functions is dense in
$\overline{L^2}(\mathbb{R}^{q},\mathbb{C})$, one can extend $\widehat{I}_{q}$
to an isometry from $\overline{L^2}(\mathbb{R}^{q},\mathbb{C})$ to
$L^{2}(\Omega)$ and the
above properties hold true for this extension.\\

\noindent{\bf Remark.} Property~(\ref{e:antisym}) of the function $f$ in $\overline{L^2}(\mathbb{R}^{q},\mathbb{C})$ together with Property~(\ref{e:w4}) of $\widehat{W}$ ensure that $\widehat{I}_{q}(f)$ is a real--valued random variable.

\section{Wiener Chaos}\label{sec:wienerchaos}

Our results are based on the expansion of the function $G$, introduced in Section~\ref{sec:intro}, in Hermite polynomials. The Hermite polynomials are
\begin{equation*}
H_{q}(x)= (-1)^{q}e^{\frac{x^{2}}{2}}\frac{d^{q}}{dx^{q}}\left( e^{-\frac{x^{2}}{2}}\right)\;,
\end{equation*}
in particular, $H_0(x)=1,H_1(x)=x,H_2(x)=x^2-1$. If $X$ is a normal random variable with mean $0$ and variance $1$, then
\[
\mathbb{E}(H_q(X)H_{q'}(X))=\int_\mathbb{R}H_q(x)H_{q'}(x)\frac{1}{\sqrt{2\pi}}\rme^{-x^2/2}\rmd x=q!\delta_{q,q'}\;.
\]
Moreover,
\begin{equation}\label{e:chaos-exp}
G(X)=\sum_{q=1}^{+\infty}\frac{c_q}{q!} H_q(X)\;,
\end{equation}
where the convergence is in $L^2(\Omega)$ and where
\begin{equation}\label{e:cq}
c_q=\mathbb{E}(G(X)H_q(X))\;.
\end{equation}
The expansion~(\ref{e:chaos-exp}) is called a Wiener chaos expansion with each
term in the chaos expansion living in a different chaos. The
expansion~(\ref{e:chaos-exp}) starts at $q=1$, since
\[
c_0=\mathbb{E}(G(X)H_0(X))=\mathbb{E}(G(X))=0\;,
\]
by assumption. The condition $\mathbb{E}(G(X)^2)<\infty$ implies
\begin{equation}\label{e:summability}
\sum_{q=1}^{+\infty}\frac{c_q^2}{q!}<\infty\;.
\end{equation}

Hermite polynomials are related to multiple integrals as follows : if
$X=\int_{\mathbb{R}}g(x)\rmd\widehat{W}(x)$ with
$\mathbb{E}(X^2)=\int_{\mathbb{R}}|g(x)|^2\,\rmd x=1$ and $g(x)=\overline{g(-x)}$
so that $X$ has unit variance and is real--valued, then
\begin{equation}\label{e:herm-integ}
H_q(X)=\widehat{I}_q(g^{\otimes q})=\int_{\mathbb{R}^q}^{''}g(x_1)\cdots g(x_q)\rmd \widehat{W}(x_1)\cdots\rmd\widehat{W}(x_q)\;.
\end{equation}

The expansion~(\ref{e:chaos-exp}) of $G$ induces a corresponding expansion of the wavelet coefficients $W_{j,k}$, namely,
\begin{equation}
  \label{eq:Wjk}
W_{j,k}=\sum_{q=1}^{+\infty}\frac{c_q}{q!} W_{j,k}^{(q)}\;,
\end{equation}
where by~(\ref{e:W2}) one has
\begin{equation}
  \label{eq:Wjkq}
W_{j,k}^{(q)}=\sum_{\ell\in\mathbb{Z}}h_j^{(K)}(\gamma_j k-\ell)H_q(X_{\ell})\;.
\end{equation}

The Gaussian sequence $\{X_{n}\}_{n\in\mathbb{Z}}$ is long--range dependent
because its spectrum at low frequencies behaves like $|\lambda|^{-2d}$ with
$d>0$ and hence explodes at $\lambda=0$. What about the processes
$\{H_q(X_{\ell})\}_{\ell}$ for $q\geq 2$? What is the behavior of the spectrum
at low frequencies? Does it explodes at $\lambda=0$? The answer depends on the
respective values of $q$ and $d$.  Let us define
\begin{equation}\label{e:qcDef}
q_c=\max\{ q\in\mathbb{N}~:~  q<1/(1-2d)\}\;,
\end{equation}
and
\begin{equation}\label{e:ldparamq}
d(q)=qd+(1-q)/2\;.
\end{equation}
One has
\begin{equation}\label{e:dq>0}
d(q)>0\quad\mbox{ if }q\leq q_c,\quad\mbox{ that is if }q<1/(1-2d)\;.
\end{equation}
The following result shows that the spectral density of
$\{H_q(X_\ell)\}_{\ell\in\mathbb{Z}}$ has a different behavior at zero
frequency depending on whether $q\leq q_c$ or $q> q_c$. It is long--range
dependent when $q\leq q_c$ and short--range dependent when $q>q_c$. We first
give a definition.
\begin{definition}
  The convolution of two locally integrable $(2\pi)$-periodic functions $g_1$
  and $g_2$ is defined as
\begin{equation}\label{eq:defconv}
(g_1\star g_2)(\lambda)=\int_{-\pi}^{\pi}g_1(u)g_2(\lambda-u)\rmd u\;.
\end{equation}
Moreover the $q$ times self-convolution of $g$ is denoted by $g^{(\star q)}$. 
\end{definition}
\begin{lemma}\label{lem:convol_spectraldensity}
  Let $q$ be a positive integer. The spectral density of
  $\{H_q(X_{\ell})\}_{\ell\in\mathbb{Z}}$ is
$$
q!f^{(\star q)}=q!(f\star\dots\star f)\;,
$$
where the spectral density $f$ of $\{X_{\ell}\}_{\ell\in\mathbb{Z}}$
is given in~(\ref{e:sdf}).
Moreover the following holds :
\begin{enumerate}[(i)]
\item If $q\leq q_c$, then $\lambda^{2d(q)}f^{(\star q)}(\lambda)$ is bounded on
  $\lambda\in(0,\pi)$ and converges to a positive number as $\lambda\downarrow0$.
\item If $q> q_c$, then $f^{(\star q)}(\lambda)$ is bounded on
  $\lambda\in(0,\pi)$ and converges to a positive number
  as $\lambda\downarrow0$.
\end{enumerate}
Hence if $q\leq q_c$, $\{H_q(X_\ell)\}_{\ell}$ has long memory with parameter
$d(q)>0$ whereas if $q>q_c$, $\{H_q(X)\}_{\ell}$ has a short--memory behavior.
\end{lemma}
\begin{proof}
By definition of $H_q$ and since $X$ has unit variance by assumption, we have
$$
\mathbb{E}(H_q(X_{\ell})H_q(X_{\ell+m}))
=q!\left(\int_{-\pi}^{\pi}f(\lambda)\rme^{\rmi \lambda m}\rmd\lambda\right)^q\;.
$$
Using the fact that, for any two locally integrable $(2\pi)$-periodic functions
$g_1$ and $g_2$, one has
$$
\int_{-\pi}^{\pi}(g_1\star g_2)(\lambda)\rme^{\rmi \lambda m}\rmd\lambda
=\int_{-\pi}^{\pi}g_1(u)\rme^{\rmi u m}\rmd u
\times\int_{-\pi}^{\pi}g_2(v)\rme^{\rmi v m}\rmd v\;,
$$
we obtain that the spectral density of $\{H_q(X_\ell)\}_{\ell}$ is $q!f^{(\star
  q)}$.  

The properties of $f^{(\star q)}$ stated in
Lemma~\ref{lem:convol_spectraldensity} are proved by induction on $q$ using
Lemma~\ref{lem:convol_powerlaw}. Observe indeed that if $\beta_1=d(q)$ and
$\beta_2=2d$, then
\[
\beta_1+\beta_2-1=2d(q)+2d-1=(2dq+1-q)+2d-1=2(q+1)d-(q+1)+1=2d(q+1)\;.
\]
\end{proof}

Now, consider the expansion of $\Delta^K Y_{\ell}=G(X_{\ell})=\sum_{q=q_0}^{+\infty}(c_q/q!) H_q(X_{\ell})$, where
\begin{equation}
  \label{e:hermiterank}
q_0=\min\{q\geq1,\,c_q\neq 0\}\;.
\end{equation}
The exponent $q_{0}$ is called the {\it Hermite rank} of $\Delta^K Y$. 

In the following, we always assume that at least one summand of $\Delta^K Y_{\ell}$ has long memory, that is, in view of Lemma~\ref{lem:convol_spectraldensity},
\begin{equation}\label{e:longmemorycondition}
q_0\leq q_c \;.
\end{equation}

\section{The result and its interpretations}\label{sec:WC}

In this section we describe the limit in distribution of the wavelet
coefficients $\{W_{j+m,k}\}_{m,k}$ as $j\to\infty$, adequately normalized, and
we interpret the limit. Recall that $W_{j+m,k}$ involves a sum of chaoses of
all order. In the limit, however, only the order $q_0$ will prevail. The
convergence of finite--dimensional distributions is denoted by
$\overset{\text{\tiny{fidi}}}{\rightarrow}$.

\begin{theorem}\label{thm:AsympDistWjk}
As $j\to\infty$, we have
\begin{equation}\label{eq::AsympDistWjk}
\left\{\gamma_{j}^{-(d(q_0)+K)}W_{j+m,k},\, m,k\in
  \mathbb{Z}\right\} \overset{\text{\tiny{fidi}}}{\rightarrow} c_{q_0}\,
(f^*(0))^{q_0/2} \; \left\{ Y^{(q_0,K)}_{m,k},\, m,k\in \mathbb{Z}\right\}\;,
\end{equation}
where for every positive integer  $q$,
\begin{equation}\label{yq}
  Y^{(q,K)}_{m,k}=
  (\gammalim_m)^{1/2}\,
  \int_{\mathbb{R}^{q}}^{''}\frac{\rme^{\rmi k\gammalim_{m}(\zeta_1+\cdots+\zeta_q)}}{(\rmi(\zeta_1+\cdots+\zeta_q))^K}\;
\frac{\widehat{h}_{\infty}(\gammalim_m(\zeta_1+\cdots+\zeta_q))}
{|\zeta_1|^{d}\cdots|\zeta_q|^{d}}
\;\rmd\widehat{W}(\zeta_1)\cdots \rmd\widehat{W}(\zeta_q)\;.
\end{equation}
\end{theorem}
This Theorem is proved in Section~\ref{sec:proofth}.\\

\noindent{\bf Interpretation of the limit.}\\

The limit distribution can be interpreted as the wavelet coefficients of a
generalized Hermite process defined below, based on the wavelet family
\begin{equation}\label{e:hinfmk}
\left\{h_{\infty,m,k}(t)=
\gammalim_m^{-1/2}h_\infty(-\gammalim_m^{-1} t +k),\,
m,k\in\mathbb{Z}\right\}\;.
\end{equation}
This wavelet family is the natural one to consider because the Fourier
transform $\widehat{h}_{\infty}(\lambda)$ is the rescaled limit of the original
$\widehat{h}_{j}(\lambda)$ as indicated in~(\ref{EqLimHj}).

A generalized process is indexed not by time but by functions. The generalized
Hermite processes for any order $q$ in $\{1,\dots,q_c\}$ are defined as
follows~:
\begin{definition}\label{def:generalK}
  Let $0<d<1/2$ and let $q$ be a positive integer such that $0<q<1/(1-2d)$ and
  $K\geq0$. Define the set of functions
$$
\mathcal{S}_{q,d}^{(K)}=\left\{\theta\;,
\int_{\mathbb{R}} \left|\widehat{\theta}(\xi)\right|^2\;|\xi|^{q-1-2d
  q-2K}\rmd\xi<\infty \right\} \;,
$$
where $\widehat{\theta}=\mathfrak{F}[\theta]$. The generalized random process
$Z_{q,d}^{(K)}$ is indexed by functions $\theta\in\mathcal{S}_{q,d}^{(K)}$ and
is defined as
\begin{equation}\label{EqHarmDerRos}
Z_{q,d}^{(K)}(\theta)=\int_{\mathbb{R}^{q}}^{\prime\prime}
\frac{\overline{\widehat{\theta}(u_1+\cdots+u_q)}}
{(\rmi(u_1+\cdots+u_q))^K|u_1\cdots u_q|^{d}}
\;\rmd\widehat{W}(u_1)\cdots \rmd\widehat{W}(u_q)\;,
\end{equation}
where $\widehat{\theta}=\mathfrak{F}[\theta]$ as defined in~(\ref{eq:TFdef}).
\end{definition}
Now fix $(m,k)\in\mathbb{Z}^2$ and choose a function
$h_{\infty,m,k}(t),t\in\mathbb{R}$ as in~(\ref{e:hinfmk}), so that
\begin{equation}\label{e:hinfiny}
\mathfrak{F}[h_{\infty,m,k}](\xi)=\mathfrak{F}[\gammalim_m^{-1/2} h_{\infty}(-\gammalim_m^{-1} t+k)](\xi)=
(\gammalim_m)^{1/2}\;\rme^{-\rmi\gammalim_m\xi}\;
\overline{\widehat{h}_\infty(\gammalim_m\xi)}\;.
\end{equation}

\begin{lemma}\label{lem:existenceZqdK}
  The conditions on $d$ and $q$ in Definition~\ref{def:generalK} ensures the
  existence of $Z_{q,d}^{(K)}(\theta)$. In particular,
\[
h_{\infty,m,k}\in \mathcal{S}_{q,d}^{(K)}\mbox{ for all }K\in\{0,\dots,M\}\;,
\]
and hence $Z_{q,d}^{(K)}(h_{\infty,m,k})$ is well-defined.
\end{lemma}
This Lemma is proved in Section~\ref{sec:appendix}.

By setting in~(\ref{EqHarmDerRos}), $\theta=h_{\infty,m,k}$, defined in~(\ref{e:hinfiny}), we obtain for all $(m,k)\in\mathbb{Z}^2$,
$$
Y^{(q,K)}_{m,k}= Z_{q,d}^{(K)}(h_{\infty,m,k})\;.
$$
Hence the right-hand side of~(\ref{eq::AsympDistWjk}) are the wavelet
coefficients of the generalized process $Z_{q,d}^{(K)}$ with respect to the wavelet family $ \{h_{\infty,m,k},\,m,k\in\mathbb{Z}\}$.

In the special case $q=1$ (Gaussian case), this result corresponds to that of
Theorem~1(b) and Remark~5 in~\cite{moulines:roueff:taqqu:2007:jtsa}, obtained
in the case where $\gamma_j=2^j$. In this special case, we have
$Z_{1,d}^{(K)}=B_{(d+K)}$, where $B_{(d)}$ is the centered generalized
Gaussian process such that for all
$\theta_1,\theta_2\in\mathcal{S}_{1,d}^{(0)}$,
$$
\mathrm{Cov}(B_{(d)}(\theta_1),B_{(d)}(\theta_2))
=\int_{\mathbb{R}} |\lambda|^{-2d}\widehat{\theta_1}(\lambda)\overline{\widehat{\theta_2}(\lambda)}\;\rmd\lambda\;.
$$

It is interesting to observe that, under additional assumptions on
$\theta$, for $K\geq1$, $Z_{q,d}^{(K)}(\theta)$ can also be defined by
\begin{equation}\label{eq:fubiniZ}
Z_{q,d}^{(K)}(\theta)=\int_{\mathbb{R}}\tilde{Z}_{q,d}^{(K)}(t)\overline{\theta(t)}\;\rmd t\;,
\end{equation}
where $\{\tilde{Z}_{q,d}^{(K)}(t),\,t\in\mathbb{R}\}$ denotes a measurable
continuous time process defined by
\begin{equation}\label{EqHarmRos}
\tilde{Z}_{q,d}^{(K)}(t)=
\int_{\mathbb{R}^{q}}^{"}
\frac{\rme^{\rmi (u_1+\cdots+u_q)\,t}-
\sum_{\ell=0}^{K-1}\frac{(\rmi(u_1+\cdots+u_q)\,t)^{\ell}}{\ell!}}
{(\rmi(u_1+\cdots+u_q))^K|u_1\cdots u_q|^{d}}\;
\rmd\widehat{W}(u_1)\cdots \rmd\widehat{W}(u_q),\,t\in\mathbb{R}\;.
\end{equation}
If, in~(\ref{eq:fubiniZ}) we set $K=1$, we recover the usual Hermite process as defined in \cite{Taq79} which has stationary increments. The process $\tilde{Z}_{q,d}^{(K)}(t)$ can be regarded as the Hermite process $\tilde{Z}_{q,d}^{(1)}(t)$ integrated $K-1$ times. In
the special case where $K=q=1$, we recover the Fractional Brownian Motion
$\{B_H(t)\}_{t\in\mathbb{R}}$ with Hurst index $H=d+1/2\in (1/2,1)$.

In the case $K=0$ we cannot define a random process $Z_{q,d}^{(0)}(t)$ as in
~(\ref{EqHarmRos}). The case $K=0$ would correspond to the derivative of the
Hermite process $\tilde{Z}_{q,d}^{(1)}(t)$ but the Hermite process is not
differentiable and thus the process $\tilde{Z}_{q,d}^{(0)}(t),t\in\mathbb{R}$
is not defined. When $K=0$ one can only consider the generalized process
$\tilde{Z}_{q,d}^{(0)}(\theta)$. Relation~(\ref{EqHarmRos}) can be viewed as
resulting from~(\ref{EqHarmDerRos}) and (\ref{eq:fubiniZ}) by interverting
formally the integral signs.

We now state sufficient conditions on $\theta$ for~(\ref{eq:fubiniZ}) to
hold.
\begin{lemma}\label{lem:fubiniZ}
  Let $q$ be a positive integer such that $0<q<1/(1-2d)$ and $K\geq1$.  Suppose
  that $\theta\in\mathcal{S}_{q,\delta}^{(K)}$ is complex valued with at least
  $K$ vanishing moments, that is,
\begin{equation}
  \label{eq:moment-condZ}
\int_{\mathbb{R}}\theta(t)\,t^\ell\;\rmd t=0\quad\text{for all}\quad\ell=0,1,\dots,K-1\;.
\end{equation}
Suppose moreover that
\begin{equation}
  \label{eq:fubini-cond}
\int_{\mathbb{R}} |\theta(t)|\,|t|^{K+(d-1/2)q}\;\rmd t < \infty \;.
\end{equation}
Then Relation~(\ref{eq:fubiniZ}) holds.
\end{lemma}
This lemma is proved in Section~\ref{sec:appendix}.

If, for example, the $h_j$ are derived from a compactly supported
multiresolution analysis then $h_{\infty}$ will have compact support and so
$h_{\infty,m,k}$ will satisfy~(\ref{eq:fubini-cond}). In this case, the limits
$Y_{m,k}^{(q,K)}$ in Theorem~\ref{thm:AsympDistWjk} can therefore be
interpreted, for $m,k\in\mathbb{Z}$ as the wavelet coefficients of the process
$Z_{q,d}^{(K)}$ belonging to the $q$--th chaos. This interpretation is a useful
one even when the technical assumption~(\ref{eq:fubini-cond}) is not satisfied.\\

\noindent{\bf Self-similarity.}\\

The processes ${Z}_{q,d}^{(K)}$ and $\tilde{Z}_{q,d}^{(K)}$ are
self-similar. Self-similarity can be defined for processes indexed by
$t\in\mathbb{R}$ as well as for generalized processes indexed by functions
$\theta$ belonging to some suitable space $\mathcal{S}$, for example the space
$\mathcal{S}_{q,\delta}^{(K)}$ defined above.

A process $\{Z(t),\,t\in\mathbb{R}\}$ is said to be \emph{self-similar with
parameter $H>0$} if for any $a>0$,
$$
\{a^H \, Z(t/a),\,t\in\mathbb{R}\}
\overset{\text{\tiny{fidi}}}{=}
\{Z(t),\,t\in\mathbb{R}\}\;,
$$ 
where the equality holds in the sense of finite-dimensional distributions.  A
generalized process $\{Z(\theta),\,\theta\in\mathcal{S}\}$ is said to be
\emph{self-similar with parameter $H>0$} if for any $a>0$ and
$\theta\in\mathcal{S}$,
$$
Z(\theta^{a,H})
\overset{\text{\tiny{d}}}{=}
Z(\theta)\;,
$$ 
where $\theta^{a,H}(u)=a^{-H}\theta(u/a)$ (see \cite{major:1984}, Page 5).
Here $\mathcal{S}$ is assumed to contain both $\theta^{a,H}$ and $\theta$.
 
Observe that the process $\{\tilde{Z}_{q,d}^{(K)}(t),\,t\in\mathbb{R}\}$, with
$K\geq1$  is  self-similar with parameter
\begin{equation}
  \label{eq:Hss}
H=K+qd-q/2=(K-1)+(d(q)+1/2)\;.  
\end{equation}
As noted above $\tilde{Z}_{q,d}^{(K)}$ can be regarded as
$\tilde{Z}_{q,d}^{(1)}$ integrated $K-1$ times.

The generalized process
$\{{Z}_{q,d}^{(K)}(\theta),\,\theta\in\mathcal{S}_{q,\delta}^{(K)}\}$, which is
defined in~(\ref{EqHarmDerRos}) with $K\geq0$, is self-similar with the same
value of $H$ as in~(\ref{eq:Hss}), but this time the formula is also valid for
$K=0$. 

In particular, the Hermite process ($K=1$) is self-similar with
$H=d(q)+1/2\in(1/2,1)$ and the generalized process $Z_{q,d}^{(0)}(\theta)$
with $K=0$ is self-similar with $H=d(q)-1/2\in(-1/2,0)$.\\

\noindent{\bf Interpretation of the result.}\\

In view of the preceding discussion, the wavelet
coefficients of the subordinated process $Y$ behave at large scales
($\gamma_j\to\infty$) as those of a self-similar process $Z_{q,d}^{(K)}$ living
in the chaos of order $q_0$ (the Hermite rank of $G$) and with self-similar
parameter $K+d(q_0)-1/2$.

\section{Proof of Theorem~\ref{thm:AsympDistWjk}}\label{sec:proofth}

{\bf Notation.} It will be convenient to use the following notation. We denote by $\Sigma_q,\,q\geq 1$, the $\mathbb{C}^{q}\to\mathbb{C}$
function defined, for all $y=(y_1,\dots,y_{q})$ by
\begin{equation}
    \label{eq:DefPartialSums}
    \Sigma_q(y)=\sum_{i=1}^{q} y_i\;.
  \end{equation}
With this notation $Y_{m,k}^{(q,K)}$ in Theorem~\ref{thm:AsympDistWjk} can be
expressed as 
\[
 Y^{(q,K)}_{m,k}=
  (\gammalim_m)^{1/2}\,
  \int_{\mathbb{R}^{q}}^{''}\frac{\exp\circ\Sigma_q(\rmi k\gammalim_{m}\zeta)}{\left(\Sigma_q(\rmi\zeta)\right)^K}
\cdot \frac{\widehat{h}_{\infty}\circ \Sigma_q(\gammalim_m\zeta)}
{|\zeta_1|^{d}\cdots|\zeta_q|^{d}}
\;\rmd\widehat{W}(\zeta_1)\cdots \rmd\widehat{W}(\zeta_q)\;.
\]
where $\circ$ denotes the composition of functions.\\

We will separate the Wiener chaos expansion~(\ref{eq:Wjk}) of $W_{j,k}$ into two terms
depending on the position of $q$ with respect to $q_c$.
The first term includes only the $q$'s for which $H_q(x)$ exhibits long--range
dependence (LD), that is,
\begin{equation}
  \label{eq:Wjklrd}
W_{j,k}^{(LD)}=\sum_{q=0}^{q_c}  \frac{c_q}{q!} W_{j,k}^{(q)}\;,
\end{equation}
and the second term includes the terms which exhibit short--range dependence (SD)
\begin{equation}
  \label{eq:Wjksd}
W_{j,k}^{(SD)}=\sum_{q=q_c+1}^{\infty} \frac{c_q}{q!} W_{j,k}^{(q)}\;.
\end{equation}

Using Representation (\ref{e:intrepX}) and~(\ref{e:herm-integ}) since $X$ has unit variance, one has
for any $\ell\in\mathbb{Z}$,
\begin{align*}
H_q(X_\ell)&=
H_q\left(\int_{-\pi}^{\pi}\rme^{\rmi\xi \ell}f^{1/2}(\xi)
\rmd\widehat{W}(\xi)\right)\\
&=\int_{(-\pi,\pi]^q}^{''}
\exp\circ\Sigma_q(\rmi\ell\xi)\times
\left(f^{\otimes q}(\xi)\right)^{1/2}
\;\rmd\widehat{W}(\xi_1)\cdots \rmd\widehat{W}(\xi_q)\;.
\end{align*}
Then by~(\ref{eq:Wjkq}),(\ref{EqHjk}) and (\ref{EqHjkVSHj}), we have
\begin{eqnarray*}
W_{j,k}^{(q)}&=&\sum_{\ell\in\mathbb{Z}}h_{j}^{(K)}(\gamma_j k-\ell) H_q(X_{\ell})\\
&=&\sum_{\ell\in\mathbb{Z}}h_{j}^{(K)}(\gamma_j k-\ell)\int_{(-\pi,\pi]^q}^{''}
\exp\circ\Sigma_q(\rmi\ell\xi)\times
\left(f^{\otimes q}(\xi)\right)^{1/2}
\;\rmd\widehat{W}(\xi_1)\cdots \rmd\widehat{W}(\xi_q)\\
&=&\int_{(-\pi,\pi]^q}^{''}\left(\sum_{\ell\in\mathbb{Z}}h_{j}^{(K)}(\gamma_j k-\ell)
\exp\circ\Sigma_q(\rmi\ell\xi)\right)
\left(f^{\otimes q}(\xi)\right)^{1/2}
\;\rmd\widehat{W}(\xi_1)\cdots \rmd\widehat{W}(\xi_q)\\
&=&\int_{(-\pi,\pi]^q}^{''}\rme^{\Sigma_q(\rmi\gamma_j k\xi)}\left(\sum_{m\in\mathbb{Z}}h_{j}^{(K)}(m)
\exp\circ\Sigma_q(-\rmi m\xi)\right)\left(f^{\otimes q}(\xi)\right)^{1/2}
\;\rmd\widehat{W}(\xi_1)\cdots \rmd\widehat{W}(\xi_q)\\
&=& \int_{(-\pi,\pi]^q}^{''}  \rme^{\Sigma_q(\rmi\gamma_j k\xi)}
\left(\widehat{h}_{j}^{(K)}\circ\Sigma_q(\xi)\right)\left(f^{\otimes q}(\xi)\right)^{1/2}
\;\rmd\widehat{W}(\xi_1)\cdots \rmd\widehat{W}(\xi_q)\;.
\end{eqnarray*}
Then
\begin{equation}\label{eq:fjk_representation}
W_{j,k}^{(q)}=\widehat{I}_q(f_{j,k}^{(q)})\;,
\end{equation}
with
\[
f_{j,k}^{(q)}(\xi)=\left(\exp\circ\Sigma_q(\rmi k\gamma_j\xi)\right)\left(
\widehat{h}_{j}^{(K)}\circ\Sigma_q(\xi)\right)
\left(f^{\otimes q}(\xi)\right)^{1/2}
\times\1_{(-\pi,\pi)}^{\otimes q}(\xi)
\;,
\]

\noindent where $\xi=(\xi_1,\cdots,\xi_q)$ and $f^{\otimes q}(\xi)=f(\xi_1)\cdots f(\xi_q)$.\\

The two following results provide the asymptotic behavior of each term of the sum
in~(\ref{eq:Wjklrd}) and of $W_{j,k}^{(SD)}$, respectively. They are proved in
Sections~\ref{sec:asympt-behav-w_jq} and~\ref{sec:short-range-depend},
respectively. The first result concerns the terms with long memory, that is, with $q\leq q_c$. The second result concerns the terms with short memory for which $q>q_c$.

\begin{proposition}\label{PropCvWjk}
  Suppose that $q\in\{1,\dots,q_c\}$. Then, as $j\to\infty$,
\begin{equation}\label{c1}
\left(\gamma_{j}^{-(d(q)+K)}W_{j+m,k}^{(q)},\, m,k\in \mathbb{Z}\right)
\overset{\text{\tiny{fidi}}}{\rightarrow}
\left((f^*(0))^{q/2} \; Y^{(q,K)}_{m,k},\, m,k\in \mathbb{Z}\right)\;,
\end{equation}
where
$
Y^{(q,K)}_{m,k}$ is given by (\ref{yq}).
\end{proposition}
\begin{proposition}\label{PropBoundWSD}
  We have, for any $k\in\mathbb{Z}$, as $j\to\infty$,
  \begin{equation}
    \label{eq:boundWSD}
    W_{j+m,k}^{(SD)}=O_P(\gamma_j^K)\;.
\end{equation}
\end{proposition}

It follows from Proposition~\ref{PropCvWjk} that the dominating term
in~(\ref{eq:Wjklrd}) is given by the chaos of order $q=q_0$.
Now, since $d(q_0)>0$ by~(\ref{e:dq>0}), we get from Proposition~\ref{PropBoundWSD} that, for all
$(k,m)$, as $j\to\infty$,
\begin{equation*}
W_{j+m,k}^{(SD)}=o_p(\gamma_{j}^{d(q_0)+K})\;.
\end{equation*}
This concludes the proof of Theorem~\ref{thm:AsympDistWjk}.

\subsection{Proof of Proposition~\ref{PropCvWjk}}
\label{sec:asympt-behav-w_jq}
We first express the distribution of $\{W_{j+m,k}^{(q)},\,m,k\in\mathbb{Z}\}$
as a finite sum of stochastic integrals and then show that each integral
converges in $L^2(\Omega)$.

\begin{lemma}\label{lem:equal_dis_wjqs}
Let $q\in\mathbb{N}^{*}$. For any $j$
\begin{equation}\label{ExpWjk}
W_{j+m,k}^{(q)}\overset{(\text{\tiny{fidi}})}{=}
\sum_{s=-[q/2]}^{[q/2]}W_{m,k}^{(j,q,s)}\;,
\end{equation}
where $[a]$ denotes the integer part of $a$, and for any $q\in\mathbb{N}^{*}$,
$s\in\mathbb{Z}$, 
\begin{equation}\label{ExpWjks}
W_{m,k}^{(j,q,s)}=\int_{\zeta\in\mathbb{R}^q}^{''}
\1_{\Gamma^{(q,s)}}(\gamma_j^{-1}\zeta)
f_{m,k}(\zeta;j,q)\;\rmd\widehat{W}(\zeta_1) \cdots \rmd\widehat{W}(\zeta_q)\;,
\end{equation}
where $f_{m,k}(\zeta;j,q)$ is defined by (setting $\xi=\gamma_j^{-1}\zeta$)
\begin{equation}\label{eq:fjk}
f_{m,k}(\gamma_j\xi;j,q)=\gamma_j^{-q/2}
\frac{\exp\circ\Sigma_q(\rmi\gamma_{j+m}k\xi)\times
\widehat{h}_{j+m}\circ\Sigma_q(\xi)}
{\{1-\exp\circ\Sigma_q(-\rmi\xi)\}^K}
\left(f^{\otimes q}(\xi)\right)^{1/2}.
\end{equation}
and where
\begin{equation}\label{EqGammas}
\Gamma^{(q,s)}=
\left\{\xi\in (-\pi,\pi]^q,\,
-\pi+2s\pi< \sum_{i=1}^q\xi_i \leq \pi+2s\pi\right\}\;.
\end{equation}
\end{lemma}
\begin{proof}
Using~(\ref{eq:fjk_representation}), with $j$ replaced by $j+m$,
and~(\ref{EqHjkVSHj}), we get
$$
W_{j+m,k}^{(q)}=
\int_{(-\pi,\pi]^q}^{''}\exp\circ\Sigma_q(\rmi\gamma_{j+m} k\xi)
\frac{\widehat{h}_{j+m}\circ\Sigma_q(\xi)}
{\{1-\exp\circ\Sigma_q(\rmi\xi\}^K}
\left(f^{\otimes q}(\xi)\right)^{1/2}\;
\rmd\widehat{W}(\xi_1)\cdots \rmd\widehat{W}(\xi_q)\;.
$$
By~(\ref{eq:fjk}), we thus get
\begin{align}
\label{eq:Wjksansscaling}
W_{j+m,k}^{(q)}
&=\displaystyle\int_{\xi\in(-\pi,\pi]^q}^{''}\gamma_j^{q/2}f_{m,k}(\gamma_j\xi;j,q)\;
\rmd\widehat{W}(\xi_1)\cdots \rmd\widehat{W}(\xi_q)\\
\nonumber
&\displaystyle\overset{(\text{\tiny{fidi}})}{=}
\int_{\zeta\in(-\gamma_j\pi,\gamma_j\pi]^q}^{''}
f_{m,k}(\zeta;j,q)\;
\rmd\widehat{W}(\zeta_1)\cdots \rmd\widehat{W}(\zeta_q)\;,
\end{align}
where we set $\zeta=\gamma_j\xi$ (see Theorem~4.4 in \cite{major:1984}).
Observe that for all $\zeta\in(-\gamma_j\pi,\gamma_j\pi]^q$,
$$
-\pi\gamma_j-2[q/2]\pi\gamma_j\leq -q\gamma_j\pi
\leq \sum_{i=1}^q\zeta_i \leq q\gamma_j\pi\leq \pi\gamma_j+2[q/2]\pi\gamma_j\;.
$$
The result follows by using that
for any $\zeta\in(-\gamma_j\pi,\gamma_j\pi]^q$, there is a unique
$s=-[q/2],\dots,[q/2]$ such that
$\zeta/\gamma_j\in\Gamma^{(q,s)}$.
\end{proof}
\medskip
\noindent{\bf Proof of Proposition~\ref{PropCvWjk}.} In view of Lemma~\ref{lem:equal_dis_wjqs}, we shall look at the $L^2(\Omega)$ convergence of the normalized $W_{m,k}^{(j,q,s)}$ at each value of $s$. Proposition~\ref{PropCvWjk} will follow from the following
convergence results, valid for all fixed $m,k\in \mathbb{Z}$ as $j\to\infty$. For $s=0$,
\begin{equation}
  \label{eq:convL2s_fixed}
  \gamma_{j}^{-(d(q)+K)}W_{m,k}^{(j,q,0)} \overset{L^2}{\rightarrow}
  (f^*(0))^{q/2} \; Y^{(q,K)}_{m,k}  \;,
\end{equation}
 whereas for other values of $s$, namely for all $s\in\{-[q/2],\dots,-1,1,\dots,[q/2]\}$,
  \begin{equation}
    \label{eq:convL2s_nz}
  \gamma_{j}^{-(d(q)+K)}W_{m,k}^{(j,q,s)} \overset{L^2}{\rightarrow}0\;,
  \end{equation}
  where $d(q)$ is defined in~(\ref{e:ldparamq}).\\

We now prove these convergence using the representation~(\ref{ExpWjks}). By~(\ref{e:sdf}) and $|1-\rme^{\rmi\lambda}|\geq2|\lambda|/\pi$ on
  $\lambda\in(-\pi,\pi)$, we have that
\begin{equation}\label{e:fbound}
f(\lambda)\leq \left(\frac\pi{2}\right)^{-2d}\, \|f^*\|_\infty
\,|\lambda|^{-2d}\,,\quad \lambda\in[-\pi,\pi]\;.
\end{equation}
By definition of $\Gamma^{(q,s)}$ in~(\ref{EqGammas}), we have, for all
$\zeta\in\gamma_j \Gamma^{(q,s)}$, $\gamma_j^{-1}\sum_i\zeta_i-2\pi s
\in(-\pi,\pi]$. Hence using the $(2\pi)$-periodicity of $\widehat{h}_{j+m}$, we
can use~(\ref{EqMajoHj}) for bounding
$\widehat{h}_{j+m}(\gamma_j^{-1}\sum_i\zeta_i)$. With the change of variables $\zeta=\gamma_j\xi$ and (\ref{e:fbound}), for all $\zeta\in\gamma_j\Gamma^{(q,s)}$ and $j$ large enough so that
$\gamma_{j+m}/\gamma_j\geq \gammalim_m/2$,
\begin{equation}
\label{eq:boundfjk}
\gamma_{j}^{-(d(q)+K)}\left|f_{m,k}(\zeta;j,q)\right|=\gamma_{j}^{-(dq-q/2+1/2+K)}\left|f_{m,k}(\zeta;j,q)\right|\leq
 C_0\,
g(\zeta;2\pi\gamma_j s)\;,
\end{equation}
where $C_0$ is a positive constant and
$$
g(\zeta;t)=
\left(1+\left|\sum_{i=1}^q\zeta_i- t\right|\right)^{-\alpha-K}\, \prod_{i=1}^{q}|\zeta_i|^{-d}\;.
$$
The squared $L^2$-norm of $g(\cdot;t)$ reads
$$
J(t)=\int_{\mathbb{R}^d} g^2(\zeta;t)\,\rmd\zeta
=\int_{\mathbb{R}^{q}}\left(1+\left|\sum_{i=1}^q\zeta_i-t\right|\right)^{-2\alpha-2K}\prod_{i=1}^{q}|\zeta_i|^{-2d}\, \prod_{i=1}^q\rmd\zeta_i\;.
$$
We now show that Lemma~\ref{LemInt} applies with $M_1=2\alpha+2K$, $M_2=0$ and $\beta_i=2d$ for $i=1,\dots,q$. Indeed, we have
$M_2-M_1=-2\alpha-2K\leq -2\alpha<-1$. Further, for all $\ell=1,\dots,q-1$, we have, by the assumption on $d$,
$$
\sum_{i=\ell}^q\beta_i=2d(1+q-\ell)>(1+q-\ell)(1-1/q) =q-\ell+(\ell-1)/q\geq q-\ell\,.
$$
Finally, since $\alpha>1/2$,  one has
$M_2-M_1+q=-2\alpha-2K+q<q-1\leq\sum_i \beta_i$.  \\

Applying Lemma \ref{LemInt}, we
get $J(t)\to0$ as $|t|\to\infty$ and $J(0)<\infty$.  Thus, if $s\neq0$, one has $t=2\pi\gamma_j s\to\infty$ as $j\to\infty$ and hence we
obtain~(\ref{eq:convL2s_nz}).  If $s=0$, then $t=2\pi\gamma_j s=0$ and using the bound~(\ref{eq:boundfjk}),
$J(0)<\infty$, and the dominated convergence theorem, we have that the convergence~(\ref{eq:convL2s_fixed}) follows from the
convergence at a.e. $\zeta\in\mathbb{R}^q$ of the left hand side of~(\ref{eq:boundfjk}), which we now
establish. Recall that $f_{m,k}$ is defined in~(\ref{eq:fjk}). By~(\ref{EqLimHj}),~(\ref{e:sdf}) and the continuity of $f^*$
at the origin, we have, as $j\to\infty$,
\begin{eqnarray*}
\gamma_{j}^{-{1}/{2}}
\widehat{h}_{j+m}\circ\Sigma_q\left(\zeta/\gamma_{j}\right)
&=&\left( \frac{\gamma _{j+m}}{\gamma _{j}}\right)^{1/2}\gamma_{j+m}^{-{1}/{2}}
\; \widehat{h}_{j+m}\circ\Sigma_q\left((\zeta/\gamma _{j+m})
(\gamma _{j+m}/\gamma _{j})\right)\\
&\to &\bar{\gamma }_{m}^{{1}/{2}}\;
\widehat{h}_{\infty}(\gammalim_{m}(\zeta_1+\cdots+\zeta_q))\;,
\end{eqnarray*}
and for every $\ell=1,\cdots,q$
\begin{equation*}
\gamma_{j} ^{-2d}f({\zeta_{l}}/{\gamma _{j}})=
\gamma_{j} ^{-2d} \left| 1-\rme^{-\rmi {\zeta_{l}}/{\gamma _{j}}}\right| ^{-2d}f^{\ast}({\zeta_{l}}/{\gamma _{j}})\to f^{\ast} (0) \vert \zeta _{l}\vert ^{-2d}\;.
\end{equation*}
Hence
$\gamma_{j}^{-(d(q)+K)}
f_{m,k}(\zeta;j,q,0)\1_{\Gamma^{(q,s)}}(\gamma_j^{-1}\zeta)$
converges to
$$
(\gammalim_m)^{1/2}(f^*(0))^{q/2}\;
\frac{\rme^{\rmi k\gammalim_{m}(\zeta_1+\cdots+\zeta_q)}\times
\widehat{h}_{\infty}(\gammalim_m(\zeta_1+\cdots+\zeta_q))}
{(\rmi(\zeta_1+\cdots+\zeta_q))^K|\zeta_1|^{d}\cdots |\zeta_q|^{d}}\;.
$$
This concludes the proof.
\null\hfill $\Box$\par\medskip

\subsection{Proof of Proposition~\ref{PropBoundWSD}}
\label{sec:short-range-depend}
We now consider the short-range dependence part of the wavelet coefficients
$(W_{j,k})$ defined by~(\ref{eq:Wjkq}) and~(\ref{eq:Wjksd}).
These wavelet coefficients can be equivalently defined as
\begin{equation}
  \label{eq:WjkSDbis}
  W_{j,k}^{(SD)}=\sum_{\ell\in\mathbb{Z}}h_j^{(K)}(\gamma_j k-\ell)\Delta^K Y_\ell^{(SD)}\;,
\end{equation}
where we have set
$$
\Delta^K Y_{\ell}^{(SD)}=\sum_{q\geq q_{c}+1} \frac{c_{q}}{q!}H_{q}(X_{\ell}),\quad \ell\in\mathbb{Z}\;.
$$
Using Lemma~\ref{lem:convol_spectraldensity}, since~(\ref{e:summability}) holds
and $\{H_q(X_\ell)\}_{\ell\in\mathbb{Z}}$ are uncorrelated weakly stationary
processes, the process $\{\Delta^K Y_{\ell}^{(SD)}\}_{\ell\in\mathbb{Z}}$ is
weakly stationary with spectral density
$$
f^{(SD)}(\lambda)=\sum_{q\geq q_{c}+1} \frac{c_{q}^{2}}{q!} 
f^{(\star q)}(\lambda),\quad \lambda\in(-\pi,\pi)\;.
$$
By Lemma~\ref{lem:convol_spectraldensity}(ii), we have that 
$\|f^{(\star \{q_c+1\})}\|_\infty<\infty$. Using that 
$\|g_1\star g_2\|_\infty\leq\|g_1\|_\infty\|g_2\|_1$ and $\|f\|_1=1$ by
assumption, an induction yields
$$
\sup_{q>q_c}\|f^{(\star q)}\|_\infty \leq \|f^{(\star \{q_c+1\})}\|_\infty\;. 
$$
Hence, by~(\ref{e:summability}), we get $\|f^{(SD)}\|_\infty<\infty$. 
It follows that, for
$W_{j,k}^{(SD)}$ defined in~(\ref{eq:WjkSDbis}), there is a positive constant
$C$ such that,
$$
\mathbb{E}[W_{j,k}^{(SD)2}] \leq 
\|f^{(SD)}\|_\infty
\int_{-\pi}^\pi|\widehat{h}^{(K)}_j(\lambda)|^2\rmd\lambda\leq C\int_0^\pi
|\lambda|^{-2K}\;|\widehat{h}_j(\lambda)|^2\rmd\lambda
=O(\gamma_j^{2K})\;,
$$
where we used~(\ref{EqMajoHj}) with $M\geq K$ and $\alpha>1/2$.
This last relation implies~(\ref{eq:boundWSD}) and concludes the proof of Proposition~\ref{PropBoundWSD}.
\null\hfill $\Box$\par\medskip

\section{Proof of Lemmas~\ref{lem:existenceZqdK} and~\ref{lem:fubiniZ}}\label{sec:appendix}
\subsection{Proof of Lemma~\ref{lem:existenceZqdK}}
Let us first prove that if $\theta\in \mathcal{S}_{q,d}^{(K)}$ then $Z_{q,d}^{(K)}(\theta)$ exists. Indeed, by Definition~\ref{def:generalK}, $Z_{q,d}^{(K)}(\theta)$ exists if
\begin{equation}\label{e:exZqdK}
\int_{\mathbb{R}^{q}}\frac{|\widehat{\theta}(u_1+\cdots+u_q)|^2}{|u_1+\cdots+u_q|^{2K}|u_1\cdots u_q|^{2d}}\rmd u_1\cdots\rmd u_q<\infty\;.
\end{equation}
Use now Lemma~\ref{LemChagtVar} with $\beta_1=\dots=\beta_q=-2d$ and $f(x)=|\widehat{\theta}(x)|^2/|x|^{2K}$
and deduce that Condition~(\ref{e:exZqdK}) is equivalent to
\begin{equation}\label{e:square}
\Gamma\int_{\mathbb{R}}|\widehat{\theta}(s)|^2|s|^{q-1-2qd-2K}\rmd s<\infty\;,
\end{equation}
where
\[
\Gamma=\prod_{i=2}^{q}\left(\int_{\mathbb{R}}|t|^{q-i-2d(q-i+1)}|1-t|^{-2d}\rmd t\right)\;.
\]
Note that the conditions $0<d<1/2$ and $0<q<1/(1-2d)$ ensure that $\Gamma$ is finite. Further, Relation~(\ref{e:square}) implies $\theta\in \mathcal{S}_{q,d}^{(K)}$.

We now prove that for any $m,k$, $h_{\infty,m,k}\in \mathcal{S}_{q,d}^{(K)}$
when $K\in\{0,\dots,M\}$. By Definition~(\ref{e:hinfiny}) of $h_{\infty,m,k}$
\[
\widehat{h}_{\infty,m,k}(\xi)=(\gammalim_m)^{1/2}\;\rme^{-\rmi\gammalim_m\xi}\;
\overline{\widehat{h}_\infty(\gammalim_m\xi)}\;.
\]
Hence
\[
\int_{\mathbb{R}}|\widehat{h}_{\infty,m,k}(s)|^2|s|^{q-1-2qd-2K}\rmd s=\gammalim_m\int_{\mathbb{R}}
|\widehat{h}_\infty(\gammalim_m s)|^2 |s|^{q-1-2qd-2K}\rmd s\;.
\]
Set $v=\gammalim_m s$ and deduce that $h_{\infty,m,k}\in \mathcal{S}_{q,d}^{(K)}$ is equivalent to
\[
\gammalim_m^{2-(q-1-2qd-2K)}\int_{\mathbb{R}}
|\widehat{h}_\infty(v)|^2 |v|^{q-1-2qd-2K}\rmd v<\infty\;.
\]
Assumption~(\ref{EqMajoHinf}) implies that
\[
\int_{\mathbb{R}}|\widehat{h}_\infty(v)|^2 |v|^{q-1-2qd-2K}\rmd v\leq \int_{\mathbb{R}}\frac{|v|^{2M}}{(1+|v|)^{2M+2\alpha}}|v|^{q-1-2qd-2K}\rmd v\;.
\]
Since $M\geq K$ and $q(1-2d)\in (0,1)$ then
$2M+q-1-2qd-2K=(2M-2K)+q(1-2d)-1>-1$. Further $\alpha>1/2$ and $q(1-2d)\in
(0,1)$ imply that
$2M-2M-2\alpha+(q-1-2qd-2K)=-2\alpha-2K+q(1-2d)-1<-1$. Then
\[
\int_{\mathbb{R}}|\widehat{h}_{\infty,m,k}(s)|^2|s|^{q-1-2qd-2K}\rmd s<\infty\;.
\]
holds and $h_{\infty,m,k}\in \mathcal{S}_{q,d}^{(K)}$.

\subsection{Proof of Lemma~\ref{lem:fubiniZ}}

Let $a_t(u_1,\cdots,u_q)$ denote the kernel of the integral in~(\ref{EqHarmRos}) defining $\tilde{Z}_{q,d}^{(K)}$ and suppose we can exchange the order of integration and write
\begin{equation}\label{e:reverse}
\int_{\mathbb{R}}\tilde{Z}_{q,d}^{(K)}(t)\theta(t)\rmd t=\int_{\mathbb{R}^q}^{''}\left[\int_{\mathbb{R}}a_t(u_1,\cdots,u_q)\theta(t)\rmd t\right]\rmd\widehat{W}(u_1)\cdots\rmd\widehat{W}(u_q)\;.
\end{equation}
Then condition~(\ref{eq:moment-condZ}) gives
$$
\int_{\mathbb{R}}
\left[\rme^{\rmi t\,(u_1+\cdots+u_q)}-
\sum_{\ell=0}^{K-1}\frac{(\rmi t(u_1+\cdots+u_q))^{\ell}}{\ell!}\right]
\,\overline{\theta(t)}\,\rmd t =
\int_{\mathbb{R}}\rme^{\rmi t\,(u_1+\cdots+u_q)}\,\overline{\theta(t)}\,\rmd t=\overline{\widehat{\theta}\circ\Sigma_q(u)}\;,
$$
showing that~(\ref{e:reverse}) equals $\tilde{Z}_{q,d}^{(K)}(\theta)$ defined in~(\ref{EqHarmDerRos}). It remains to justify the change of order of integration in~(\ref{e:reverse}) by using a stochastic Fubini theorem, (see for instance
\cite[Theorem~2.1]{pipiras:taqqu2010}). A sufficient condition is
\[
\int_{\mathbb{R}}\left(a_t^2(u_1,\cdots,u_q)\rmd u_1\cdots\rmd u_q\right)^{1/2}\rmd t<\infty\;.
\]
This condition is satisfied, because setting $v=tu$, we have
\[
\int_{\mathbb{R}^q}
\left|\rme^{\rmi t\,(u_1+\cdots+u_q)}-
\sum_{\ell=0}^{K-1}\frac{(\rmi t (u_1+\cdots+u_q))^{\ell}}{\ell!}\right|^2|\rmi(u_1+\cdots+u_q)|^{-2K}|u_1\cdots u_q|^{-2d}\;\rmd^q u\;,
\]
\[
\leq |t|^{2K+2d-q}
\int_{\mathbb{R}^q}(1+|u_1+\cdots+u_q|)^{-2K}|u_1\cdots u_q|^{-2d}\rmd^q u\;.
\]
\section{Auxiliary lemmas}\label{sec:techlemma}
The following lemma provides a bound for the convolution of two functions
exploding at the origin and decaying polynomially at infinity.
\begin{lemma}\label{lem:convolR}
Let $\alpha>1$ and $\beta_1,\beta_2\in[0,1)$ such that $\beta_1+\beta_2<1$,
and set
$$
g_i(t)=|t|^{-\beta_i} (1+|t|)^{\beta_i-\alpha}\;.
$$
Then
\begin{equation}\label{eq:convolR}
\sup_{u\in\mathbb{R}}
\left((1+|u|)^{\alpha}\;
\int_{\mathbb{R}}g_1(u-t)g_2(t)\,\rmd t\right)
<\infty\;.
  \end{equation}
\end{lemma}
\begin{proof}
We first show that
$$
J(u)=\int_{\mathbb{R}}g_1(u-t)g_2(t)\,\rmd t
=\int_{\mathbb{R}}|u-t|^{-\beta_1}(1+|u-t|)^{\beta_1-\alpha}
|t|^{-\beta_2}(1+|t|)^{\beta_2-\alpha}\,\rmd t
$$
is uniformly bounded on $\mathbb{R}$. Using the assumptions on
$\beta_1,\beta_2$, there exist $p>1$ such that
$\beta_1<1/p<1-\beta_2$. Let $q$ be such that $1/p+1/q=1$.
The H\"{o}lder inequality implies that
$$
J(u)^{pq}\leq
\int_{\mathbb{R}}|t|^{-p\beta_1}(1+|t|)^{p\beta_1-p\alpha}\,\rmd t\times
\int_{\mathbb{R}}|t|^{-q\beta_2}(1+|t|)^{q\beta_2-q\alpha}\,\rmd t\;.
$$
The condition on $\alpha,\beta_1,\beta_2,p$ and the definition of $q$ imply
that these two integrals are finite. Hence $\sup_u J(u)<\infty$.

We now determine  how fast $J(u)$ tends to 0 as $u\to\infty$.
Observe that, if
$|t-u|\leq|u|/2$, then $|t|\geq|u|/2$.  By splitting the integral in two
integrals on the domains $|t-u|\leq|u|/2$ and ${|t-u|>|u|/2}$, we get $J(u)\leq
J_1(u)+J_2(u)$ with
$$
J_1(u)\leq
(|u|/2)^{-\beta_2}(1+|u|/2)^{\beta_2-\alpha}
\int_{\mathbb{R}}|u-t|^{-\beta_1}(1+|t-u|)^{\beta_1-\alpha}\rmd t\;,
$$
and
$$
J_2(u)\leq
(|u|/2)^{-\beta_1}(1+|u|/2)^{\beta_1-\alpha}
\int_{\mathbb{R}}|t|^{-\beta_2}(1+|t|)^{\beta_2-\alpha}\rmd t\;.
$$
Now, as $|u|\to\infty$, we have $J_i(u)=O(|u|^{-\alpha})$ for $i=1,2$, which
achieves the proof.
\end{proof}

The next lemma describes the convolutions of two periodic functions that
explode at the origin as a power. A different definition of
convolution is involved here (see~(\ref{eq:defconv})).

\begin{lemma}\label{lem:convol_powerlaw}
  Let $(\beta_1,\beta_2)\in(0,1)^2$. Let $g_1$, $g_2$ be ($2\pi$)-periodic
  functions such that $g_i(\lambda)= |\lambda|^{-\beta_i}\;g_i^\ast(\lambda)$,
  $i=1,2$. Each $g_i^\ast(\lambda)$ is a $(2\pi)$-periodic non-negative
  function, bounded on $(-\pi,\pi)$ and positive at the origin, where it is
  also continuous. Let $g=g_1\star g_2$ as defined
  in~(\ref{eq:defconv}). Then,
\begin{itemize}
\item If $\beta_1+\beta_2<1$, $g$ is bounded and continuous on $(-\pi,\pi)$, and
  satisfies $g(0)>0$.
\item If $\beta_1+\beta_2>1$,
  \[
  g(\lambda)=|\lambda|^{-(\beta_1+\beta_2-1)}g^*(\lambda)\;,
  \]
  where $g^*(\lambda)$ is bounded on
  $(-\pi,\pi)$ and converges to a positive constant as $\lambda\to 0$.
  If moreover for some $\beta\in (0,2]$ such that $\beta<\beta_1+\beta_2-1$ and some $L>0$, one has for any $i\in\{1,2\}$
  \begin{equation}
    \label{eq:HbetaL}
  |g_i^*(\lambda)-g_i^*(0)|\leq L|\lambda|^\beta,\,\forall \lambda\in (-\pi,\pi)\;,
  \end{equation}
  then there exists some $L'>0$ depending only on
$L,\beta_1,\beta_2$ such that
  \[
  |g^*(\lambda)-g^*(0)|\leq L'|\lambda|^\beta,\,\forall \lambda\in (-\pi,\pi)\;.
  \]
\end{itemize}
\end{lemma}
\begin{proof}
By~(\ref{eq:defconv}) and $(2\pi)$-periodicity, we may write
\begin{equation}
  \label{eq:convolPeriodicFormula}
g(\lambda)=\int_{-\pi}^\pi g_1(u)g_2(\lambda-u)\;\rmd u
=\int_{-\pi}^\pi|\{\lambda-u\}|^{-\beta_1}g^*_1(\lambda-u)\;
|u|^{-\beta_2}g^*_2(u)\;
\rmd u \;.
\end{equation}

Let us first consider the case $\beta_1+\beta_2<1$.  We clearly have $g(0)>0$.
To prove that $g$ is bounded, we proceed as in the case of convolutions of
non-periodic functions (see the proof of Lemma~\ref{lem:convolR}), namely, for
$p,q$ such that $\beta_1<1/p<1-\beta_2$ and $1/p+1/q=1$, the H\"older
inequality gives that
\begin{equation}
  \label{eq:boundsommeBetainf1}
\|g\|_\infty^{pq}\leq\|g_1\|_p^{p}\;\|g_2\|_q^{q}
\leq\|g_1^{*}\|_\infty^{p}\;
\|g_2^*\|_\infty^{q}
\int_{-\pi}^\pi|t|^{-p\beta_1}\,\rmd t\times
\int_{-\pi}^\pi|t|^{-q\beta_2}\,\rmd t<\infty\;.
\end{equation}
For any $\epsilon>0$ and $i=1,2$,
let ${g}_{\epsilon,i}$ be the $(2\pi)$-periodic function such that
for all $\lambda\in(-\pi,\pi)$,
${g}_{\epsilon,i}(\lambda)=\1_{(-\epsilon,\epsilon)}(\lambda)
\;g_i(\lambda)$ and let $\bar{g}_{\epsilon,i}=g_i-{g}_{\epsilon,i}$.
Then $g=\bar{g}_{\epsilon,1}\star\bar{g}_{\epsilon,2}+
{g}_{\epsilon,1}\star\bar{g}_{\epsilon,2}+
\bar{g}_{\epsilon,1}\star{g}_{\epsilon,2}+
{g}_{\epsilon,1}\star{g}_{\epsilon,2}$.
Since $\bar{g}_{\epsilon,i}$ is bounded for $i=1,2$, we have
that $\bar{g}_{\epsilon,1}\star\bar{g}_{\epsilon,2}$ is continuous.
On the other hand, using the H\"older
inequality as in~(\ref{eq:boundsommeBetainf1}),
we get that $\|{g}_{\epsilon,1}\star\bar{g}_{\epsilon,2}\|_\infty$,
$\|\bar{g}_{\epsilon,1}\star{g}_{\epsilon,2}\|_\infty$,
$\|\bar{g}_{\epsilon,1}\star\bar{g}_{\epsilon,2}\|_\infty$ tend to zero
as $\epsilon\to0$. Hence $g$ is continuous as well.

We now consider the case $\beta_1+\beta_2\geq1$.
Setting $v=u/\lambda$
in~(\ref{eq:convolPeriodicFormula}), we get, for any
$\lambda\in[-\pi,\pi]\backslash\{0\}$,
$$
g^*(\lambda)=|\lambda|^{\beta_1+\beta_2-1}g(\lambda)=
\int_{\mathbb{R}}
\1_{(-\pi/|\lambda|,\pi/|\lambda|)}(v)\;
|\{(1-v)\}_\lambda|^{-\beta_1}|v|^{-\beta_2}g^*_1(\lambda(1-v))g^*_2(\lambda
v)\;\rmd v\;,
$$
where for any real number $x$ and $\lambda\neq0$, $\{x\}_\lambda$ denotes the
unique element of $[-\pi/|\lambda|,\pi/|\lambda|]$ such that
$x-\{x\}_\lambda\in\mathbb{Z}$. Take now $|\lambda|$ small enough so that
$\pi/|\lambda|>2$. Then, for any $v\in(-\pi/|\lambda|+1,\pi/|\lambda|]$, we
have $|\{(1-v)\}_\lambda|=|1-v|\geq|1-|v||$ and, for any
$v\in(-\pi/|\lambda|,-\pi/|\lambda|+1]$, we have
\begin{equation}
  \label{eq:periodicpartBound}
|\{(1-v)\}_\lambda|=
|1-v-2\pi/|\lambda||=
2\pi/|\lambda|+v-1\geq -v-1 = |1-|v||\;.
\end{equation}
Thus we have
$\1_{(-\pi/|\lambda|,\pi/|\lambda|)}(v)\; |\{(1-v)\}_\lambda|^{-\beta_1}|\leq
|1-|v||^{-\beta_1}$ for all $v\in\mathbb{R}$. We conclude that for $|\lambda|$
small enough, the integrand in the last display is bounded from above by
$|1-|v||^{-\beta_1}|v|^{-\beta_2}\;\|g^*_1\|_\infty\|g^*_2\|_\infty$, which is
integrable on $v\in\mathbb{R}$. Hence $g^*$ is bounded, and by dominated
convergence, as $\lambda\to0$,
\begin{equation}
  \label{eq:gstar}
g^*(\lambda)\to g^*_1(0)g^*_2(0)\;
\int_{\mathbb{R}}|1-v|^{-\beta_1} \;|v|^{-\beta_2}\;\rmd v>0\;.
\end{equation}
We set $g^*(0)$ equal to this limit.

Suppose moreover that $g_1^*, g_2^*$ satisfy~(\ref{eq:HbetaL}).
We take $g_1^*(0)= g_2^*(0)=1$ without loss of generality and denote
$r_i(\lambda)=|g_i^*(\lambda)-1|$ for $i=1,2$.
Then $r(\lambda)=\left|g^*(\lambda)-g^*(0)\right|$, where $g^*(0)$ is defined
as the limit in~(\ref{eq:gstar}) , is at most
$$
\int_{\mathbb{R}}
\left|
\1_{(-\pi/|\lambda|,\pi/|\lambda|)}(v)\;
|\{(1-v)\}_\lambda|^{-\beta_1}|v|^{-\beta_2}g^*_1(\lambda(1-v))g^*_2(\lambda
v)
-|1-v|^{-\beta_1} \;|v|^{-\beta_2}\right|
\;\rmd v\;.
$$
Setting $g_i^*(\lambda)=(g_i^*(\lambda)-1)+1$, we have $r\leq A+B_1+B_2+C$
with
\begin{align*}
A(\lambda)&=
\int_{\mathbb{R}}
\left|
\1_{(-\pi/|\lambda|,\pi/|\lambda|)}(v)\;
|\{(1-v)\}_\lambda|^{-\beta_1}|v|^{-\beta_2}
-|1-v|^{-\beta_1} \;|v|^{-\beta_2}\right|\;\rmd v\;,\\
B_i(\lambda)&=
\int_{\mathbb{R}}
\1_{(-\pi/|\lambda|,\pi/|\lambda|)}(v)\;
|\{(1-v)\}_\lambda|^{-\beta_{j}}|v|^{-\beta_i}r_i(\lambda v)\;\rmd v\;,
\end{align*}
where $(i,j)$ is $(1,2)$ or $(2,1)$, and
$$
C(\lambda)=
\int_{\mathbb{R}}
\1_{(-\pi/|\lambda|,\pi/|\lambda|)}(v)\;
|\{(1-v)\}_\lambda|^{-\beta_1}|v|^{-\beta_2}r_1(\lambda(1-v))r_2(\lambda
v)\;\rmd v\;.
$$
Since $\{(1-v)\}_\lambda=1-v$ for
$v\in[-\pi/|\lambda|+1,\pi/|\lambda|)$ and $\lambda$ large enough, we have
\begin{multline*}
A(\lambda)=\int_{(-\pi/|\lambda|,\pi/|\lambda|)^c}
|1-v|^{-\beta_1} \;|v|^{-\beta_2}\;\rmd v\\
+\int_{-\pi/|\lambda|}^{-\pi/|\lambda|+1}
\left|
|\{(1-v)\}_\lambda|^{-\beta_1}|v|^{-\beta_2}
-|1-v|^{-\beta_1} \;|v|^{-\beta_2}\right|\;\rmd v\;.
\end{multline*}
The first integral is $O(|\lambda|^{\beta_1+\beta_2-1})$.
Using~(\ref{eq:periodicpartBound}), the second line of the last display is less
than
$$
\int_{\pi/|\lambda|-1}^{\pi/|\lambda|}
\left[|1-v|^{-\beta_1}|v|^{-\beta_2}+
|1+v|^{-\beta_1} \;v^{-\beta_2}\right]\;\rmd v=O(|\lambda|^{\beta_1+\beta_2})\;.
$$
We conclude that as $\lambda\to0$, $A(\lambda)=O(|\lambda|^{\beta_1+\beta_2-1})$.
Moreover using that $r_i(\lambda)\leq L|\lambda|^{\beta}$ and $\beta_1+\beta_2-\beta>1$, we have
$B_i(\lambda)=O(|\lambda|^{\beta})$ for $i=1,2$. The same is true for $C$ since
$r_1$ and $r_2$ are also bounded on $\mathbb{R}$.
This achieves the proof.
\end{proof}

\begin{lemma}\label{LemChagtVar}
Let $p$ be a positive integer and $f:\mathbb{R}\to\mathbb{R}_+$. Then, for any $\beta\in\mathbb{R}^{q}$,
  \begin{equation}
    \label{eq:chagtVar}
\int_{\mathbb{R}^q}f(y_1+\cdots+y_q)\;
\prod_{i=1}^{q}|y_i|^{\beta_i}\; \rmd y_1\cdots\rmd y_q=\Gamma\times
\int_{\mathbb{R}} f(s)|s|^{q-1+\beta_1+\cdots+\beta_q}
\rmd s\;,
  \end{equation}
where, for all $i\in\{1,\cdots,q\}$, $B_i=\beta_i+\cdots+\beta_{q}$ and
$$
\Gamma=\prod_{i=2}^{q}
\left(\int_{\mathbb{R}}|t|^{q-i+B_i}|1-t|^{\beta_{i-1}}\rmd t\right)\;.
$$
(We note that $\Gamma$ may be infinite in which case~(\ref{eq:chagtVar}) holds
with the convention $\infty\times0=0$).
\end{lemma}

\begin{proof}
Relation~(\ref{eq:chagtVar}) is obtained by using the following two successive change of variables followed by an application of the Fubini
  Theorem.  Setting, for all $i=1,\cdots,q$, $u_i=\sum_{j=i}^{q}y_j$, we get that $y_i=u_i-u_{i+1}$ for $i<q$ and $y_q=u_q$. Then
  the integral in the left--hand side of~(\ref{eq:chagtVar}) reads
\begin{equation}\label{e:part1}
\int_{\mathbb{R}^q}f(u_{1})
\left[|u_{q}|^{\beta_{q}}
\prod_{i=1}^{q-1}|u_i-u_{i+1}|^{\beta_i}
\right]\;\rmd u_1\cdots\rmd u_q\;.
\end{equation}
The second change of variables consists in setting, for all $i=1,\cdots,q$,
$u_i=\prod_{j=1}^i t_j$. Then
\[
\rmd u_1\cdots\rmd u_q=\left(\prod_{i=1}^{q-1}t_i^{q-i}\right)\rmd t_1\cdots\rmd t_q,
\]
\[
\prod_{i=1}^{q-1}|u_i-u_{i+1}|^{\beta_i}=\prod_{i=1}^{q-1}\left(|t_1\cdots t_i|^{\beta_i}|1-t_{i+1}|^{\beta_i}\right)=\left(\prod_{i=1}^{q-1}|t_i|^{\beta_i+\cdots+\beta_{q-1}}\right)
\left(\prod_{i=2}^{q}|1-t_i|^{\beta_{i-1}}\right),
\]
and $|u_q|=\prod_{i=1}^{q}|t_i|^{\beta_q}$, so that (\ref{e:part1}) becomes
\[
\int_{\mathbb{R}^q}f(t_1)\prod_{i=1}^{q}|t_i|^{\beta_i+\cdots+\beta_{q}+q-i}\prod_{i=2}^{q}|1-t_i|^{\beta_{i-1}}\rmd t_1\cdots\rmd t_q\,,
\]
which by Fubini Theorem yields the required result.

\end{proof}

\begin{lemma}\label{LemInt}
  Let $a\in\mathbb{R}$ and $q$ be a positive integer.
  Let $\beta=(\beta_1,\cdots,\beta_q)\in (-\infty,1)^q$,
  $M_1>0$ and $M_2>-1$ such that $M_2-M_1<-1$. Assume that
  $q+M_2-M_1<\sum_{i=1}^{q}\beta_i$, and that for any
  $\ell\in\{1,\cdots,q-1\}$, $\sum\limits_{i=\ell}^{q} \beta_i> q-\ell$.  Set for any $a\in\mathbb{R}$,
\[
J_q(a;M_1,M_2;\beta)=\int_{\mathbb{R}^{q}}
\frac{|\Sigma_q(\zeta)-a |^{M_2}}
{(1+|\Sigma_q(\zeta)-a|)^{M_1}\prod\limits_{i=1}^{q}|\zeta_i|^{\beta_i}}
\;\rmd\zeta.
\]
Then one has
\begin{equation}\label{e:Jq1}
\sup_{a\in\mathbb{R}}(1+|a|)^{1-q+\sum_{i=1}^q\beta_i}J_q(a;M_1,M_2;\beta)<\infty\;.
\end{equation}
In particular,
\[
J_q(0;M_1,M_2;\beta)<\infty,
\]
and
\[
J_q(a;M_1,M_2;\beta)=O(|a|^{-(1-q+\sum_{i=1}^q\beta_i)})\quad\mbox{ as }a\to\infty\;.
\]
\end{lemma}
\begin{proof}
Since $J_q(a;M_1,M_2;\beta_1,\cdots,\beta_q)=J_q(-a;M_1,M_2;\beta)$, we may suppose $a\geq0$. By Lemma~\ref{LemChagtVar},
\[
J_q(a;M_1,M_2;\beta_1,\cdots,\beta_q)=\Gamma\int_{\mathbb{R}}
\frac{|s-a|^{M_2}|s|^{q-1-(\beta_1+\cdots+\beta_q)}}{(1+|s-a |)^{M_1}}\rmd s
\]
where
\[
\Gamma=\prod_{i=2}^{q}\int_{\mathbb{R}}\frac{\rmd t}{|t|^{\beta_i+\cdots+\beta_q-(q-i)}|1-t|^{\beta_{i-1}}}\;.
\]
The conditions on $\beta_i$'s, $M_1$ and $M_2$ imply
$J_q(a;M_1,M_2;\beta_1,\cdots,\beta_q)<\infty$ for all $a$. To obtain the sup on $a>0$, we
set $v=s/a$. Then, denoting $S=\sum_{i=1}^q\beta_i$, we get
\begin{equation}\label{e:Jq2}
J_q(a;M_1,M_2;\beta)= C a^{q+M_2-S}\; \int_{\mathbb{R}} |v-1|^{M_2}(1+a|v-1|)^{-M_1} |v|^{-S+(q-1)} \rmd v\;,
\end{equation}
where $C$ is a positive constant. We separate the integration domain in two. Suppose first that
$|v-1|\leq a^{-1}$. Then in this case we have $(1+a|v-1|)^{-M_1}\leq 1$. Since $|v|$ is bounded on the interval $|v-1|<a^{-1}$ for $a$ large then as $a\to\infty$,
$$
\int_{|v-1|\leq a^{-1}} |v-1|^{M_2}(1+a|v-1|)^{-M_1} |v|^{-S+(q-1)} \rmd v = O\left(\int_{|v-1|\leq a^{-1}}|v-1|^{M_2}\rmd v\right)=O(a^{-1-M_2})\;.
$$
Now suppose that $|v-1>a^{-1}|$. Then $(1+a|v-1|)^{-M_1}\leq (a|v-1|)^{-M_1}$, and
\[
\begin{array}{lll}
I&=&\int_{|v-1|> a^{-1}} |v-1|^{M_2}(1+a|v-1|)^{-M_1} |v|^{-S+(q-1)} \rmd v \\
&\leq&a^{-M_1} \int_{|v-1|> a^{-1}} |v-1|^{M_2-M_1}|v|^{-S+(q-1)} \rmd v\\
&=&a^{-M_1}\left(\int_{|v|\geq 2}|v-1|^{M_2-M_1}|v|^{-S+(q-1)} \rmd v+\int_{1/2\leq |v|\leq 2,|v-1|> a^{-1}}|v-1|^{M_2-M_1}|v|^{-S+(q-1)} \rmd v\right)\\
&&+a^{-M_1}\int_{|v|\leq 1/2,|v-1|> a^{-1}}|v-1|^{M_2-M_1}|v|^{-S+(q-1)} \rmd v\;.
\end{array}
\]
The first integral concentrates around $v=\infty$, the second around $v=1$ and the third around $v=0$. The first integral is bounded, the second is
\[
O(\int_{|v-1|> a^{-1}}|v-1|^{M_2-M_1}\rmd v)=O(a^{M_1-M_2-1}),\quad\mbox{ as }a\to\infty,
\]
and the third is bounded. Therefore we get
\[
I=O(a^{-M_1})+O(a^{-M_2-1}),
\]
since $M_2-M_1<-1$. Thus (\ref{e:Jq2}) gives
\[
J_q(a;M_1,M_2;\beta)=O(a^{-1+q-S})\quad\mbox{ as }a\to\infty,
\]
yielding the bound~(\ref{e:Jq1}).
\end{proof}

{\bf Acknowledgements.} Murad~S.Taqqu was supported in part by the NSF grants DMS--0608669 and DMS--1007616 at Boston University.
\bibliographystyle{elsarticle-num}
\bibliography{lrd}
\end{document}